\documentclass[a4paper,10pt]{amsart}

\usepackage[english]{babel}
\usepackage[utf8]{inputenc}

\usepackage{amssymb}
\usepackage{amsmath}
\usepackage{amsthm}
\usepackage{bbm}

\usepackage{enumerate}
\usepackage{tikz}
\usepackage{marginnote}
\usepackage{color}
\usepackage{orcidlink}

\newcommand{\bbC}{\mathbb{C}}

\newcommand{\bbN}{\mathbb{N}}

\newcommand{\bbR}{\mathbb{R}}

\newcommand{\calL}{\mathcal{L}}

\DeclareMathOperator{\id}{id} 
\DeclareMathOperator{\re}{Re} 
\newcommand{\argument}{\mathord{\,\cdot \,}} 
\newcommand{\dx}{\mathrm{d}} 
\newcommand{\norm}[1]{\left\lVert #1 \right\rVert} 
\newcommand{\modulus}[1]{\left\lvert #1 \right\rvert} 

\DeclareMathOperator{\trace}{tr}
\newcommand{\topInt}[1]{\operatorname{int}\left(#1\right)}
\newcommand{\conv}{\operatorname{conv}}

\newcommand{\spec}{\sigma} 
\newcommand{\spb}{s} 

\newcommand{\rightProof}{``$\Rightarrow$''\ }
\newcommand{\leftProof}{``$\Leftarrow$''\ }
\newcommand{\subsetProof}{``$\subseteq$\ ''}
\newcommand{\supsetProof}{``$\supseteq$\ ''}

\newenvironment{psmallmatrix}
	{%
		\left(
		\begin{smallmatrix}
	}%
	{%
		\end{smallmatrix}
		\right)
	}

\theoremstyle{definition}
\newtheorem{definition}{Definition}[section]
\newtheorem{remark}[definition]{Remark}

\newtheorem{example}[definition]{Example}

\theoremstyle{plain}
\newtheorem{proposition}[definition]{Proposition}
\newtheorem{lemma}[definition]{Lemma}
\newtheorem{theorem}[definition]{Theorem}

\numberwithin{equation}{section}

\begin{document}

\title{Eventual cone invariance revisited}
\author[Jochen Glück]{Jochen Glück \orcidlink{0000-0002-0319-6913}}
\address[Jochen Glück]{%
	Bergische Universit\"at Wuppertal,
	Fakult\"at f\"ur Mathematik und Naturwissenschaften,
	Gaußstr.\ 20,
	42119 Wuppertal,
	Germany
}
\email{glueck@uni-wuppertal.de}

\author[Julian Hölz]{Julian Hölz \orcidlink{0000-0001-5058-9210}}
\address[Julian Hölz]{%
	Bergische Universit\"at Wuppertal,
	Fakult\"at f\"ur Mathematik und Naturwissenschaften,
	Gaußstr.\ 20,
	42119 Wuppertal,
	Germany
}
\email{hoelz@uni-wuppertal.de}


\date{\today}

\begin{abstract}
	We consider finite-dimensional real vector spaces $X$
	ordered by a closed cone $X_+$ with non-empty interior
	and study eventual nonnegativity of matrix semigroups $(e^{tA})_{t \ge 0}$ with respect to this cone.

	Our first contribution is the observation that, for general cones, 
	one needs to distinguish between different notions of eventual nonnegativity:
	(i) uniform eventual nonnegativity means that $e^{tA}$ maps $X_+$ into $X_+$ for all sufficiently large times $t$;
	(ii) individual eventual nonnegativity means that
	for each $x \in X_+$ the vector $e^{tA}x$ is in $X_+$ for all $t$ larger than an $x$-dependent time $t_0$;
	and (iii) weak eventual nonnegativity means that for each $x \in X_+$ and each functional $x'$ in the dual cone $X'_+$
	the value $\langle x', e^{tA} x \rangle$ is in $[0,\infty)$ for all $t$ larger than an $x$- and $x'$-dependent time $t_0$.
	Until now, only the first of these notions has been studied in the literature.
	We demonstrate by examples that, somewhat surprisingly for finite-dimensional spaces, all three notions are different.

	Our second contribution is to show that typical Perron--Frobenius like properties remain valid under the weakest
	of the above notions.

	Third, we study a strengthened form of the above mentioned concepts, namely eventual positivity.
	We prove that the uniform, individual and weak versions of this property are
	-- in contrast to the nonnegative case -- equivalent,
	and that they can be characterized by spectral properties.
\end{abstract}

\maketitle

\section{Introduction and main concepts}

Let $X$ be a finite-dimensional real vector space,
ordered by a closed (and convex) cone $X_+$ with non-empty interior
(see the beginning of Section~\ref{sec:cones-finite-dim-spaces-operator-spaces} for details).
If $A: X \to X$ is linear and the one-parameter semigroup ${(e^{tA})}_{t \geq 0}$ on $X$ leaves $X_+$ invariant,
then this semigroup is said to be \emph{nonnegative}.
This property has been studied for a long time; 
it can be characterized in terms of \emph{cross-positivity} of $A$ \cite[Theorem~3 on p.\;512]{SchneiderVidyasagar1970}
and is related to stability of linear ODEs \cite[Theorem~1.4 on p.\;68]{Stern1982}.

A more subtle phenomenon is the following:
it might happen that the operators $e^{tA}$ leave $X_+$ invariant for all sufficiently
larges times $t$, but not necessarily for small $t$.
This \emph{eventual nonnegativity} or \emph{eventual cone invariance} was first investigated
for matrix powers (rather than matrix exponential functions)
in the case where $X_+$ is the standard cone in $\bbR^d$;
see for instance~\cite{Friedland1978, JohnsonTaragaza2004, McDonaldPaparellaTsatsomeros2014, ZaslavskyMcDonald2003} for just a small sample of the literature on this topic.
For matrix exponentials rather than matrix powers, eventual nonnegativity
was analyzed in~\cite{NoutsosTsatsomeros2008}.
Eventual nonnegativity with respect to general cones --
which is the topic of the present article --
was studied by Kasigwa and Tsatsomeros in~\cite{KasigwaTsatsomeros2017}
and by Sootla in~\cite{Sootla2019}.

\subsection*{Individual vs.\ uniform behaviour}

Precisely speaking, there are at least two canonical ways how one could define
eventual nonnegativity of ${(e^{tA})}_{t \geq 0}$:
one could require that there exists a $t_0 \ge 0$ such that $e^{tA} X_+ \subseteq X_+$ for all $t \ge t_0$;
or one could require that, for each $x \in X_+$, there exists $t_0 \ge 0$ such that
$e^{tA}x \in X_+$ for all $t \ge 0$.
We call the first property \emph{uniform} since $t_0$ does not depend on the initial value $x$;
this type of eventual nonnegativity was studied in \cite{KasigwaTsatsomeros2017}.
The second property is an \emph{individual} one since $t_0$ is allowed to depend on $x$.

The motivation of the present article is the observation that both notions do not coincide, in general
(Example~\ref{exa:uniform-vs-individual}).
It might come as a surprise that this distinction is necessary in finite dimensions 
and this suggests a more thorough investigation of this and related phenomena in the rest of the article.

\subsection*{Versions of eventual nonnegativity and eventual positivity}

Let us first collect the two types of eventual nonnegativity that we mentioned above,
along with a third and even weaker notion, in the following definition.
For undefined notation and terminology we refer
to the beginning of Section~\ref{sec:cones-finite-dim-spaces-operator-spaces}.

\begin{definition}[Eventual nonnegativity]
	\label{def:eventual-nonnegativity}
	Let $X$ be a finite-dimensional real vector space, ordered by a closed cone $X_+$ with non-empty interior.
	Let $A: X \to X$ be linear.
	The semigroup $(e^{tA})_{t \ge 0}$ is called \dots
	\begin{enumerate}[\upshape (a)]
		\item
		      \emph{uniformly eventually nonnegative} if there exists $t_0 \ge 0$ such that
		      $e^{tA} X_+ \subseteq X_+$ for each $t \ge t_0$.

		\item
		      \emph{individually eventually nonnegative} if for each $x \in X_+$ there exists $t_0 \ge 0$ such that
		      $e^{tA} x \in X_+$ for all $t \ge t_0$.

		\item
		      \emph{weakly eventually nonnegative} if for each $x \in X_+$ and each $x' \in X'_+$
		      there exists $t_0 \ge 0$ such that $\langle x', e^{tA} x \rangle \ge 0$ for each $t \ge t_0$.
	\end{enumerate}
\end{definition}

In addition to those eventual nonnegativity concepts,
it is natural to also define similar eventual positivity notions.

\begin{definition}[Eventual positivity]
	\label{def:eventual-positivity}
	Let $X$ be a finite-dimensional real vector space, ordered by a closed cone $X_+$ with non-empty interior.
	Let $A \in X \to X$ be linear.
	The semigroup $(e^{tA})_{t \ge 0}$ is called \dots
	\begin{enumerate}[\upshape (a)]
		\item
		      \emph{uniformly eventually positive} if there exists $t_0 \ge 0$ such that
		      $e^{tA}x \in \topInt{X_+}$ for each $0 \not= x \in X_+$ and each $t \ge t_0$.

		\item
		      \emph{individually eventually positive} if for each $0 \not= x \in X_+$ there exists $t_0 \ge 0$ such that
		      $e^{tA} x \in \topInt{X_+}$ for all $t \ge t_0$.

		\item
		      \emph{weakly eventually positive} if for each $0 \not= x \in X_+$ and each $0 \not= x' \in X'_+$
		      there exists $t_0 \ge 0$ such that $\langle x', e^{tA} x \rangle > 0$ for each $t \ge t_0$.
	\end{enumerate}
\end{definition}

Recall that the cone $X_+$ is called \emph{polyhedral} if it is an intersection
of finitely many closed half spaces in $X$.
A cone is polyhedral if and only if it is the nonnegative span of a finite number of vectors,
see~\cite[Theorem 3.37 on p.\;140]{AliprantisTourky2007}.
Moreover, the dual cone of a polyhedral cone with non-empty interior is also polyhedral~\cite[Theorem 3.25 on p.\;132]{AliprantisTourky2007}.
This easily implies that the three concepts introduced in Definition~\ref{def:eventual-nonnegativity}
are equivalent if $X_+$ is polyhedral.
Hence, cones that are not polyhedral are the most interesting case throughout the paper.

Before we continue, a brief terminological remark is in order.

\begin{remark}
	There are different naming conventions when working with operators that leave a cone invariant:

	Our usage of the notions \emph{nonnegative} and \emph{positive} follows the terminology 
	that is commonly used in finite-dimensional Perron--Frobenius theory, 
	where a vector $x$ that satisfies $x \ge 0$ is called \emph{nonnegative}.
	On the other hand, in the literature on (infinite-dimensional) ordered vector spaces and vector lattices, 
	\emph{positive} typically refers to the property that is called \emph{nonnegative} 
	in the present article 
	(and notions such as \emph{strongly positive} or \emph{strictly positive} are used 
	for what we refer to as \emph{positive}).
\end{remark}

\subsection*{Further related literature}

During the last decade the theory of eventually nonnegative matrix powers
and matrix semigroups was
complemented by a corresponding theory on infinite-dimensional function spaces,
starting with the article \cite{DanersGlueckKennedy2016}. 
In this infinite-dimensional setting it was shown in \cite[Examples~5.7 and~5.8]{DanersGlueckKennedy2016} 
that one has to distinguish between individual and uniform eventual behaviour
(but those examples are very different from the finite-dimensional examples 
that we present in Section~\ref{sec:counterexamples}).
The infinite-dimensional theory has applications to the study of parabolic partial differential equations,
as can e.g.\ be seen in \cite[Section~3.2]{AddonaGregorioRhandiTacelli2022}, \cite[Section~5]{Arora2022}, and \cite[Section~7]{DenkKunzePloss2021}.
For a recent overview of the current state of research in this field we refer to \cite{Glueck2022}.

\subsection*{Organization of the article}

In Section~\ref{sec:cones-finite-dim-spaces-operator-spaces} we give a sufficient condition for a cone in a finite-dimensional space to be closed and we analyze the cone of positive operators.
Section~\ref{sec:counterexamples} gives finite-dimensional counterexamples that show that individual eventual nonnegativity does not imply uniform eventual nonnegativity and that weak eventual nonnegativity does not imply individual eventual nonnegativity.
In Section~\ref{sec:spectral-properties-weakly-eventually-nonnegative-case} we prove that under the weakest eventually nonnegative notion a Perron--Frobenius type result holds, which yields that the spectral bound $\spb(A)$ is an eigenvalue corresponding to a nonnegative eigenvector.
We conclude in Section~\ref{sec:eventual-positivity} by showing that uniform, individual and weak eventual positivity coincide for one-parameter semigroups on finite-dimensional ordered vector spaces.
In the appendix we discuss a few observations about the closedness of cones in finite dimensions which complement
a result in Section~\ref{sec:cones-finite-dim-spaces-operator-spaces}.

\section{Properties of cones in finite dimensions}
\label{sec:cones-finite-dim-spaces-operator-spaces}

\subsection*{Notation and terminology}

Let $X$ be a finite-dimensional real vector space, ordered by a closed cone $X_+$ with non-empty interior.%
\footnote{
	Throughout we endow all finite dimensional real vector spaces $X$ with the topology 
	that is induced by any norm. 
	We often tactily endow $X$ with a norm, 
	the specific choice of which does not matter. 
	
} 
A subset $X_+ \subseteq X$ is called a \emph{wedge} if it is non-empty and satisfies $X_+ + X_+ \subseteq X_+$ and $\lambda X_+ \subseteq X_+$ for all $\lambda \geq 0$. If, in addition, $X_+$ is \emph{pointed}, meaning that $X_+ \cap(-X_+) = \{0\}$, then $X_+$ is called a \emph{cone}. A wedge $X_+$ is called \emph{generating} if $X_+ - X_+ = X$.

Let $X_+ \subseteq X$ be a wedge. 
Then the subset
\begin{align*}
	X_+' := \{ x' \in X' \mid  \langle x', x \rangle \geq 0 \text{ for all } x \in X_+  \}
\end{align*}
of the dual space $X'$ is called the \emph{dual wedge} of $X_+$.
As $X$ is finite-dimensional it follows that the wedge $X_+$ is generating if and only if 
$X'_+$ is a cone \cite[Theorem~2.13(2) on p.\;71]{AliprantisTourky2007}. 
If $X_+$ is closed, then the dual result is also true: 
$X_+$ is a cone if and only if the dual wedge is generating; 
the implication ``$\Rightarrow$'' follows e.g.\ from \cite[Theorem~2.13(2) on p.\;71]{AliprantisTourky2007}, 
and the converse implication follows easily from the definitions.
It is well-known (and not difficult to show) that a wedge $X_+$ is generating 
if and only if the topological interior $\topInt{X_+}$ of $X_+$ is non-empty, see, e.g.,~\cite[Lemma~3.2 on p.\;119]{AliprantisTourky2007}.

\subsection*{A condition for the closedness of a cone}

The following sufficient condition for the closedness of cones generated by dilation invariant sets 
will be useful for the proof of Proposition~\ref{prop:operator-dual-cone}.
We discuss the limitations of Lemma~\ref{lem:closed-convexification} in the appendix.

\begin{lemma}
	\label{lem:closed-convexification}
	Let $X$ be a finite-dimensional real vector space and
	let $E \subseteq X$ be a closed set that satisfies $\lambda E \subseteq E$
	for all numbers $\lambda \geq 0$.
	Assume that $\conv(E) \cap -\conv(E) = \{0\}$ (or more generally that $\conv(E) \cap -E = \{0\}$).
	Then $\conv(E)$ is closed.
\end{lemma}

\begin{proof}
	Set $d := \dim(X)+1$.
	By Caratheodory's theorem, see \cite[Theorem IV.17.1 on p.\;155]{Rockafellar1997}, 
	each vector in $\conv(E)$ is a convex combination of at most $d$ vectors from $E$. 
	Together with the assumption $[0,\infty)E \subseteq E$ this implies that the linear map
	\begin{align*}
		T : \, X^d \to X, 
		\quad
		y \mapsto y_1 + \dots + y_d
	\end{align*}
	maps $E^d$ surjectively to $\conv(E)$.
	Now consider a sequence ${(x^{(n)})}_{n \in \bbN}$ in $\conv(E)$ that converges to a point $x \in X$ 
	and choose a sequence $(y^{(n)})_{n \in \bbN}$ in $E^d$ such that $Ty^{(n)} = x^{(n)}$ for each $n \in \bbN$.
	We only need to show that the sequence $(y^{(n)})_{n \in \bbN}$ is bounded in $X^d$; 
	then it has a subsequence that converges to a point $y \in E^d$ (due to the closedness of $E$) 
	and hence $x = Ty \in \conv(E)$.
	
	To prove that $(y^{(n)})_{n \in \bbN}$ is indeed bounded, assume the contrary 
	and fix an arbitrary norm on the finite-dimensional vector space $X^d$ for the following argument.
	After replacing $(y^{(n)})_{n \in \bbN}$ 
	-- and, accordingly, $(x^{(n)})_{n \in \bbN}$ -- 
	with a subsequence we can achieve that $0 < \norm{y^{(n)}} \to \infty$ 
	and that $(y^{(n)} / \norm{y^{(n)}} )_{n \in \bbN}$ converges to a vector $y \in X^d$. 
	Since $[0,\infty)E \subseteq E$ and $E$ is closed, one has $y \in E^d$.
	Moreover, $\norm{y} = 1$, so at least one component of $y$, say $y_1$, is non-zero. 
	On the other hand, $Ty = \lim_n x_n / \norm{y_n} = 0$, so it follows that 
	\begin{align*}
		-y_1 = y_2 + \dots + y_d \in \conv E \cap -E = \{0\},
	\end{align*}
	which is a contradiction.
\end{proof}

\subsection*{Cones in finite-dimensional operator spaces}

Let $X$ be a finite-dimensional real vector space, ordered by a closed cone $X_+$ with non-empty interior.
Denote the space of all linear mappings from $X$ to $X$ by $\calL(X)$. 
We endow $\calL(X)$ with the usual operator norm that is induced by a fixed norm on $X$. 

A mapping $T \in \calL(X)$ is called \emph{nonnegative} if $TX_+ \subseteq X_+$.
Let us denote the set of those mappings by $\calL(X)_+$.
Clearly, $\calL(X)_+$ is closed in $\calL(X)$ and if the cone $X_+$ is generating in $X$
one can readily check that $\calL(X)_+$ is itself a cone.
In the following we characterize the interior of this cone (Proposition~\ref{prop:operator-cone-interior})
and we describe its dual cone (Proposition~\ref{prop:operator-dual-cone}).

To this end we need the following notation.
For every $x \in X$ and $x' \in X'$ we define the operator $x \otimes x' \in \calL(X)$ by $(x \otimes x') y := \langle x', y \rangle x$ for all $y \in X$.
We note that $x \otimes x'$ has operator norm $\norm{x} \norm{x'}$ and its trace is equal to $\langle x', x \rangle$.
If $x$ and $x'$ are both non-zero, then $x \otimes x'$ is a rank-$1$ operator and otherwise the operator is $0$.

\begin{proposition}
	\label{prop:operator-cone-interior}
	Let $X$ be a finite-dimensional real vector space, ordered by a closed cone $X_+$ with non-empty interior.
	Then the interior of the cone $\calL(X)_+$ in $\calL(X)$ is non-empty
	and consists of precisely those operators $T \in \calL(X)$ that satisfy
	\begin{align*}
		T(X_+ \setminus \{0\})
		\subseteq
		\topInt{X_+}.
	\end{align*}
	In particular, for all $x \in \topInt{X_+}$ and $x' \in \topInt{X_+'}$
	the operator $x \otimes x'$ is in the interior of $\calL(X)_+$ and the cone $\calL(X)_+$ is generating.
\end{proposition}

\begin{proof}
	First note that, if the claimed equivalence is true,
	then the last claim in the proposition follows immediately:
	for all $x \in \topInt{X_+}$, $x' \in \topInt{X_+'}$, and $y \in X_+ \setminus \{0\}$,
	one has $\langle x', y \rangle > 0$
	(since the mapping $\langle \argument, y \rangle: X' \to \bbR$ is open)
	and thus $(x \otimes x')y = \langle x', y \rangle x \in \topInt{X_+}$.
	In particular, $\topInt{\mathcal{L}(X)_+}$ is non-empty since $\topInt{X_+}$ and $\topInt{X'_+}$ are non-empty.
	So by~\cite[Lemma~3.2 on p.\;119]{AliprantisTourky2007} the cone $\mathcal{L}(X)_+$ is indeed generating.

	Now fix $T \in \calL(X)$.
	We need to show that
	$T \in \topInt{\mathcal{L}(X)_+}$ if and only if $T(X_+ \setminus \{0\}) \subseteq \topInt{X_+}$.

	\rightProof
	Let $T \in \topInt{\mathcal{L}(X)_+}$ and $x \in X_+ \setminus \{0\}$.
	Then there exists an $\varepsilon >0$ such that, whenever $S \in \calL(X)$ with $\norm{S} < \varepsilon$, we have $T + S \in \calL(X)_+$. 
	Let $y \in X$ with $\norm{y} < \varepsilon \norm{x}$. 
	We claim that $Tx + y \in X_+$, which shows that $Tx \in \topInt{X_+}$.
	
	By a corollary of the Hahn-Banach extension theorem, see~\cite[Corollary 2.3.4]{Pedersen1989}, 
	there exists a continuous linear functional $x' : X \to \bbR$ of norm $1$ such that $\langle x', x \rangle = \norm{x}$.
	Define an operator $S := \frac{1}{\norm{x}} y \otimes x'$. 
	Then $Sx = y$ and $\norm{S} < \varepsilon$,
	so $T + S \in \calL(X)_+$ and it follows that $Tx + y = (T + S)x \in X_+$, as claimed.

	\leftProof
	Let $T \in \mathcal{L}(X)$ satisfy $T(X_+ \setminus \{0\}) \subseteq \topInt{X_+}$. 
	Denote the intersection of $X_+$ with the unit sphere in $X$ by $A$.
	As $A$ is compact and $T$ is continuous, $TA$ is a compact subset of $\topInt{X_+}$.
	Thus we can find an $\varepsilon > 0$ such that the open $\varepsilon$-neighbourhood of $TA$ 
	is contained in $\topInt{X_+}$.
	
	Now consider an operator $S \in \calL(X)$ of norm $\norm{S} < \varepsilon$.
	For every normalized vector $x \in X_+$ one has $\norm{Sx} < \varepsilon$, 
	so the distance of $(T+S)x$ to $TA$ is strictly smaller than $\varepsilon$ and thus, 
	$(T+S)x \in \topInt{X_+}$.
	Hence, we have $(T + S)x \in X_+$ for all $x \in X_+$, which shows that $T + S \in \calL(X)_+$.
	So $T$ is indeed an interior point of $\calL(X)_+$.
\end{proof}

Let the finite-dimensional real vector space $X$ be endowed with a closed cone $X_+$ with non-empty interior. 
We now describe the dual cone $\calL(X)'_+$.
To this end, we need a specific type of linear functionals $\calL(X) \to \bbR$:
for all vectors $x \in X$ and $x' \in X'$ we define the functional $\varphi_{x,x'}$ on $\calL(X)$
by
\begin{align*}
	\langle \varphi_{x,x'}, M \rangle := \langle x', M x \rangle = \trace \big( M (x \otimes x') \big)
\end{align*}
for all $M \in \calL(X)$,
where the operator $x \otimes x' \in \calL(X)$ was defined before Proposition~\ref{prop:operator-cone-interior}. 
Here $\trace(M)$ denotes the trace of the operator $M$.
It is not difficult to check that $\norm{\varphi_{x,x'}} = \norm{x} \norm{x'}$.
By representing all elements of $\calL(X)$ as matrices one can easily see that
the set of all functionals $\varphi_{x,x'}$ spans the dual space $\calL(X)'$ of $\calL(X)$. 

\begin{proposition}
	\label{prop:operator-dual-cone}
	Let $X$ be a finite-dimensional real vector space, ordered by a closed cone $X_+$ with non-empty interior.
	Then the dual cone of $\calL(X)_+$ in $\calL(X)'$ is the convex hull of the set
	\begin{align*}
		E := \big\{ \varphi_{x,x'}: \; x \in X_+ \text{ and } x' \in X'_+ \big\}.
	\end{align*}
\end{proposition}

\begin{proof}
	Let us first show that $\overline{\conv}(E) = \calL(X)_+'$. 
	
	\subsetProof
	It follows readily from the definition of the functionals $\varphi_{x,x'}$ that $E \subseteq \calL(X)_+'$.
	Since $\calL(X)_+'$ is convex and closed it follows that $\overline{\conv}(E) \subseteq \calL(X)_+'$. 
	
	\supsetProof
	Suppose to the contrary that we can find a functional $\varphi \in \calL(X)_+'$ that is not in $\overline{\conv}(E)$.
	By the Hahn-Banach separation theorem there exists $M \in \calL(X) (\simeq \calL(X)'')$ such that
	\begin{equation*}
		\langle \varphi, M \rangle < a := \inf \{ \langle \psi, M \rangle \mid  \psi \in \overline{\conv}(E) \}.
	\end{equation*}
	The set $\{ \langle \psi, M \rangle \mid  \psi \in \overline{\conv}(E) \}$ contains $0$ 
	(as $E$ contains the zero functional) and is stable under multiplication with nonnegative scalars; 
	as the set is bounded below by $\langle \varphi, M \rangle$, it follows that $a=0$.
	
	This implies that $\langle x', Mx \rangle \geq 0$ for all $x \in X_+$ and all $x' \in X_+'$, 
	so $M \in \calL(X)_+$.
	By the positivity of $\varphi$ we obtain $\langle \varphi, M \rangle \geq 0$, 
	which is a contradiction;
	this completes the proof of ``$\subseteq$''.
	
	It remains to show that $\conv(E)$ is closed.
	To this end we only need to check that $E$ satisfies the assumptions of Lemma~\ref{lem:closed-convexification}.
	Clearly, $E$ is closed under multiplication with nonnegative scalars. 
	Moreover, since $\conv(E)$ is a subset of $\calL(X)'_+$ and the latter set is a cone 
	(as $\calL(X)_+$ has non-empty interior, see Proposition~\ref{prop:operator-cone-interior}), 
	it follows that $\conv(E) \cap -\conv(E) = \{0\}$.
	
	To show the closedness of $E$, let $\varphi_{x_n, x'_n}$ be a sequence in $E$ 
	that converges to a functional $\psi \in \calL(X)'$. 
	As $\norm{\varphi_{x_n, x'_n}} = \norm{x_n} \norm{x_n'}$ for all $n$,
	we can rescale all the vectors $x_n$ and $x'_n$ such that the functionals $\varphi_{x_n, x'_n}$ do not change  
	but both sequences $(x_n)$ and $(x'_n)$ become bounded. 
	After switching to subsequences we may thus assume that $(x_n)$ converges to a vector $x \in X_+$ 
	and that $(x'_n)$ converges to a vector $x' \in X'_+$. 
	This yields $\psi = \varphi_{x,x'} \in E$.
\end{proof}

\section{Counterexamples}
\label{sec:counterexamples}

In this section we give several counterexamples to show that it is indeed necessary
to distinguish the different versions of eventual nonnegativity in Definition~\ref{def:eventual-nonnegativity}.
Example~\ref{exa:uniform-vs-individual} demonstrates that individual eventual nonnegativity
does not imply its uniform counterpart,
and Example~\ref{exa:individual-vs-weak} shows that weak eventual nonnegativity
does not imply the individual property. 

Moreover, we show in Example~\ref{exa:ind-semigroup-and-dual} that both a semigroup and its dual 
can be individually eventually nonnegative without being uniformly eventually nonnegative.
On the other hand, individual eventual nonnegativity is not respected by dualization in general, 
as we demonstrate in Example~\ref{exa:weak-indiviual-nonnegativity-doesnot-dualize}
-- in contrast to the weak and the uniform case, see Proposition~\ref{prop:dualize}.

\begin{example}[Uniform vs.\ individual eventual nonnegativity]
	\label{exa:uniform-vs-individual}
	There exists a closed cone $\bbR^4_+ \subseteq \bbR^4$ with non-empty interior
	and a matrix $A \in \bbR^{4 \times 4}$ such that
	the semigroup $(e^{tA})_{t \ge 0}$ is individually but not uniformly eventually nonnegative.

	Let $\bbR^4_+$ be the linearly transformed ice cream cone that is given by
	\begin{align*}
		\bbR^4_+
		 & :=
		\Big\{x \in \bbR^4 : \; (x_2 - x_1)^2 + x_3^2 + x_4^2 \le x_1^2 \text{ and } x_1 \ge 0 \Big\} \\
		 & =
		\Big\{x \in \bbR^4 : \; x_2^2  + x_3^2 + x_4^2  \le  2 x_1 x_2  \text{ and } x_1 \ge 0 \Big\}.
	\end{align*}
	Note that the following properties hold for every $x \in \bbR^4_+$: 
	(i) we have $x_1 = 0$ if and only if $x = 0$; 
	(ii) one has $x_2 \ge 0$; 
	and (iii) $x_2 = 0$ if and only all three components $x_2,x_3,x_4$ are equal to $0$.

	The bijective linear map
	\begin{align*}
		\begin{pmatrix}
			x_1 \\ x_2 \\ x_3 \\ x_4
		\end{pmatrix}
		\mapsto
		\begin{pmatrix}
			x_1 \\ x_2 - x_1 \\ x_3 \\ x_4
		\end{pmatrix}
	\end{align*}
	maps $\bbR^4_+$ onto the usual ice cream cone
	\begin{align*}
		\bbR^4_{\operatorname{ICE}} := \Big\{ x \in \bbR^4 : \; x_2^2 + x_3^2 + x_4^2 \le x_1^2  \text{ and }  x_1 \ge 0 \Big\}.
	\end{align*}
	Hence, $\bbR^4_+$ is isomorphic to the usual ice cream cone;
	in particular, $\bbR^4_+$ is indeed a closed cone with non-empty interior.

	Now we consider the block diagonal matrices
	\begin{align*}
		A :=
		\begin{pmatrix}
			0 & 1 &                         &              \\
			0 & 0 &                         &              \\
			  &   & 0                       & 2 \cdot 2\pi \\
			  &   & -\frac{1}{2} \cdot 2\pi & 0
		\end{pmatrix}
		, 
		\quad 
		e^{tA} =
		\begin{pmatrix}
			1 & t &                          &               \\
			0 & 1 &                          &               \\
			  &   & \cos(2\pi t)             & 2\sin(2\pi t) \\
			  &   & -\frac{1}{2}\sin(2\pi t) & \cos(2\pi t)
		\end{pmatrix}
		,
	\end{align*}
	where the formula for $e^{tA}$ for each $t \ge 0$
	follows by using that
	\begin{align*}
		\begin{pmatrix}
			0                       & 2 \cdot 2\pi \\
			-\frac{1}{2} \cdot 2\pi & 0
		\end{pmatrix}
		=
		\begin{pmatrix}
			2 & 0 \\
			0 & 1
		\end{pmatrix}
		\begin{pmatrix}
			0     & 2\pi \\
			-2\pi & 0
		\end{pmatrix}
		\begin{pmatrix}
			\frac{1}{2} & 0 \\
			0           & 1
		\end{pmatrix}.
	\end{align*}
	Let us show that $(e^{tA})_{t \ge 0}$ is individually eventually nonnegative 
	with respect to the cone $\bbR^4_+$:

	Fix $x \in \bbR^4_+$.
	If $x = 0$ there is nothing to show, so let $x \not= 0$ and thus $x_1 \not= 0$.
	If $x_2 = 0$, then $x$ is a multiple of the first canonical unit vector
	and hence $e^{tA}x = x \in \bbR^4_+$ for all times $t \ge 0$.
	So let us now consider the case where both $x_1,x_2$ are positive.
	For each $t \ge 0$ the vector
	\begin{align*}
		y(t) := e^{tA}x
	\end{align*}
	satisfies $y_1(t) = x_1 + tx_2$ and $y_2(t) = x_2$.
	Moreover the components $y_3(t)$ and $y_4(t)$ are bounded as $t \to \infty$.
	Thus, the inequality
	\begin{align*}
		y_2(t)^2 + y_3(t)^2 + y_4(t)^2 \le 2 y_1(t) y_2(t)
	\end{align*}
	is satisfied for all sufficiently large times $t$.
	This shows that $y(t) \in \bbR^4_+$ for all sufficiently large $t$.

	Finally we show that the semigroup is not uniformly eventually nonnegative 
	with respect to the cone $\bbR^4_+$.
	For each integer $n \ge 1$ we consider the vector
	\begin{align*}
		x(n) :=
		\begin{pmatrix}
			n \\ 1 \\ 0 \\ \sqrt{2n-1}.
		\end{pmatrix}
	\end{align*}
	One readily checks that $x(n) \in \bbR^4_+$.
	At the times $t_n := n - \frac{3}{4}$ we have $\sin(2\pi t_n) = 1$ and $\cos(2\pi t_n) = 0$, and thus
	\begin{align*}
		e^{t_n A} x(n)
		=
		\begin{pmatrix}
			n + t_n       \\
			1             \\
			2 \sqrt{2n-1} \\
			0
		\end{pmatrix}.
	\end{align*}
	Using this explicit formula, one can now check by a brief computation that the vector $z := e^{t_n A} x(n)$ 
	satisfies $2 z_1 z_2 < z_2^2 + z_3^2 + z_4^2$,
	so $e^{t_n A} x(n) \not\in \bbR^4_+$.
	Hence, we found a sequence of vectors $x(n)$ in $\bbR^4_+$ and a sequence of times $t_n \to \infty$
	such that $e^{t_n A} x(n) \not\in \bbR^4_+$ for each $n$.
	So $(e^{tA})_{t \ge 0}$ is not uniformly eventually nonnegative.
\end{example}

We do not know whether there also exists a counterexample in dimension~$3$.

\begin{remark}
	In~\cite[Proposition~1]{Sootla2019} it was claimed that individual and uniform
	eventual nonnegativity are equivalent.%
	\footnote{
		Note that this is stated there using different wording
		since the terminology in~\cite{Sootla2019} differs from ours.
	}
	Example~\ref{exa:uniform-vs-individual} shows that this is not correct.%
	\footnote{
		This was also kindly confirmed to us by the author of~\cite{Sootla2019}.
	}

	At first glance this appears to cause complications for some of the arguments in~\cite{Sootla2019} 
	that rely on spectral results for eventually nonnegative matrix semigroups
	from~\cite{KasigwaTsatsomeros2017},
	as the latter were established in~\cite{KasigwaTsatsomeros2017}
	only for the case of uniform eventual nonnegativity.
	This can be resolved, though, as we show
	in Section~\ref{sec:spectral-properties-weakly-eventually-nonnegative-case}
	that the same spectral results remain true
	for the individually (and in fact, even the weakly) eventually nonnegative case.
\end{remark}

\begin{example}[Individual vs.\ weak eventual nonnegativity]
	\label{exa:individual-vs-weak}
	There exists a finite-dimensional real vector space $Y$ ordered by a closed cone $Y_+$ with non-empty interior
	and an operator $B \in \calL(Y)$ such that the semigroup $(e^{tB})_{t \ge 0}$ is
	weakly but not individually eventually nonnegative.

	Indeed, let $X$ be a finite-dimensional real vector space ordered by a closed cone $X_+$ with non-empty interior
	and $A \in \calL(X)$ an operator such that the semigroup $(e^{tA})_{t \ge 0}$ is
	individually but not uniformly eventually nonnegative.
	Such objects exist according to Example~\ref{exa:uniform-vs-individual}.

	Consider the space $Y := \calL(X)$ and endow it with the cone $Y_+ := \calL(X)_+$ of nonnegative operators.
	Let $B \in \calL(Y)$ be defined by left multiplication with $A$, i.e.\
	\begin{align*}
		B(M) := AM
	\end{align*}
	for each $M \in \calL(X)$.
	Then $e^{tB}(M) = e^{tA}M$ for each $t \ge 0$ and each $M \in \calL(X)$.

	The matrix semigroup $(e^{tB})_{t \ge 0}$ is not individually eventually nonnegative,
	since $e^{tB} \id_X = e^{tA}$ is not eventually in $\calL(X)_+$.
	However, the semigroup $(e^{tB})_{t \ge 0}$ is weakly eventually nonnegative.
	To see this, let $M \in \calL(X)_+$, $x \in X_+$, and $x' \in X'_+$.
	Consider the functional $\varphi_{x,x'}$ on $\calL(X)$ introduced before Proposition~\ref{prop:operator-dual-cone}:
	one has
	\begin{align*}
		\langle \varphi_{x,x'}, e^{tB}(M) \rangle_{Y' \times Y} = \langle x', e^{tA} M x \rangle_{X' \times X},
	\end{align*}
	and the latter value is in $[0,\infty)$ for all sufficiently large times $t$
	since $e^{tA}Mx$ is in $X_+$ for all sufficiently large $t$.
	According to Proposition~\ref{prop:operator-dual-cone} the dual cone in $\calL(X)'$ is the convex hull
	of all such functionals $\varphi_{x,x'}$.
	This implies that
	\begin{align*}
		\langle \varphi, e^{tB}(M) \rangle_{Y' \times Y} \ge 0
	\end{align*}
	for all sufficiently large $t$ whenever $\varphi \in \calL(X)'_+$ and $M \in \calL(X)_+$.
	Hence, the semigroup $(e^{tB})_{t \ge 0}$ is indeed weakly eventually nonnegative.
\end{example}

Clearly, the space $Y$ in the above example can be chosen to have dimension $16$  
(as the space in Example~\ref{exa:uniform-vs-individual} has dimension $4$).
We do not know an example in smaller dimension.

Next we analyze how eventual nonnegativity behaves with respect to dualization. 
For the uniform and the weak case the answer is simple and is given in the following proposition. 
The proof is straightforward, so we omit it.

\begin{proposition}
	\label{prop:dualize}
	Let $X$ be a finite-dimensional real vector space, ordered by a closed cone $X_+$ with non-empty interior.
	Let $A \in \calL(X)$.
	\begin{enumerate}[\upshape (a)]
		\item\label{prop:dualize:itm:uniform}
		The semigroup $(e^{tA})_{t \ge 0}$ is uniformly eventually nonnegative
		if and only if the same is true for its dual semigroup.

		\item\label{prop:dualize:itm:weak}
		The semigroup $(e^{tA})_{t \ge 0}$ is weakly eventually nonnegative
		if and only if the same is true for its dual semigroup.
	\end{enumerate}
\end{proposition}

Let us demonstrate now that the situation is more involved for individual eventual nonnegativity. 
We first show for the individually (but not uniformly) eventually nonnegative semigroup 
from Example~\ref{exa:uniform-vs-individual} that the dual semigroup is individually eventually nonnegative, too 
(Example~\ref{exa:ind-semigroup-and-dual}). 
Hence, individual eventual nonnegativity of both a semigroup and its dual 
does not imply uniform eventual nonnegativity. 

Afterwards we show for the weakly (but not individually) eventually nonnegative semigroup 
from Example~\ref{exa:individual-vs-weak} that the dual semigroup is even individually eventually nonnegative 
(Example~\ref{exa:weak-indiviual-nonnegativity-doesnot-dualize}).
Hence, individual eventual nonnegativity is not preserved by taking (pre-)duals.

\begin{example}[Individual eventual nonnegativity of a semigroup and its dual]
	\label{exa:ind-semigroup-and-dual}
	There exists a finite-dimensional real vector space $X$ ordered by a closed cone $X_+$ with non-empty interior and an operator $A \in \calL(X)$ such that the semigroup $(e^{tA})_{t \geq 0}$ and its dual semigroup $(e^{tA'})_{t \geq 0}$ are individually eventually nonnegative but $(e^{tA})_{t \geq 0}$ is not uniformly eventually nonnegative.

	Endow $X := \bbR^4$ with the cone $X_+ := \bbR^4_+$ from Example~\ref{exa:uniform-vs-individual} 
	and let $A \in \bbR^{4 \times 4}$ denote the matrix from this example.
	Then the dual semigroup $(e^{tA'})_{t \ge 0}$ is also individually eventually nonnegative.
	
	To see this, we identify the dual space $(\bbR^4)'$ with $\bbR^4$ in the usual way. 
	By using the isomorphism between $\bbR^4_+$ and $\bbR^4_{\operatorname{ICE}}$
	given in Example~\ref{exa:uniform-vs-individual} and the fact that $\bbR^4_{\operatorname{ICE}}$ is self-dual 
	under the usual identification of $\bbR^4$ with $(\bbR^4)'$, see~\cite[Introduction to Chapter~2 on p.\;377]{Loewy1975},
	it is straightforward to check that the bijective linear map
	\begin{align*}
		\begin{pmatrix}
			x_1 \\ x_2 \\ x_3 \\ x_4
		\end{pmatrix}
		\mapsto
		\begin{pmatrix}
			x_1 + x_2 \\ x_2 \\ x_3 \\ x_4
		\end{pmatrix}
	\end{align*}
	maps the dual cone $X'_+$ of $\bbR^4_+$ onto $\bbR^4_{\operatorname{ICE}}$, 
	and thus
	\begin{align*} 
		X'_+ 
		& =
		\left\{ 
			x \in \mathbb{R}^4 
			\mid 
			x_2^2 + x_3^2 + x_4^2 \leq (x_1+x_2)^2 \text{ and } x_1 + x_2 \ge 0 
		\right \}
		\\ 
		& =
		\left\{ 
			x \in \mathbb{R}^4 
			\mid 
			x_3^2 + x_4^2 \leq x_1^2 + 2 x_1 x_2 \text{ and } x_1 \geq -x_2 
		\right \}
		.
	\end{align*}
	Now fix a vector $x \in X'_+$ and consider its trajectory $y$ under the dual semigroup, 
	given by $y(t) := e^{tA'}x$ for each $t \ge 0$.
	By the formula for $e^{tA}$ in Example~\ref{exa:uniform-vs-individual} one has
	\begin{align*}
		e^{tA'} 
		=
		\begin{pmatrix}
			1 & 0 &               &                          \\
			t & 1 &               &                          \\
			  &   & \cos(2\pi t)  & -\frac{1}{2}\sin(2\pi t) \\
			  &   & 2\sin(2\pi t) & \cos(2\pi t)
		\end{pmatrix}
	\end{align*}
	for each $t \ge 0$. 
	If $x_1 = 0$, then $x$ is a multiple of the second unit vector and thus $y(t) = x \in X'_+$ for all $t \ge 0$.
	If $x_1 < 0$, then $x_2 > 0$, and multiplying the inequality $x_1 \ge -x_2$ with $x_1$ yields $x_1^2 \leq -x_1 x_2$, 
	from which we deduce
	\begin{align*}
		0 \leq x_3^2 + x_4^2 \leq x_1^2 + 2 x_1 x_2 \leq x_1 x_2 < 0,
	\end{align*}
	a contradiction. 
	So we may assume that $x_1 > 0$.
	Then for large $t \geq 0$ the inequalities
	\begin{align*}
		y_1(t)^2 + 2 y_1(t) y_2(t)
		=
		x_1^2 + 2 t x_1^2 + 2 x_1 x_2
		\ge
		y_3(t)^2 + y_4(t)^2
	\end{align*}
	and
	\begin{align*}
		y_1(t) + y_2(t) = (1 + t) x_1 + x_2 \geq 0
	\end{align*}
	hold, since the right hand side of the first inequality is bounded.
	So $y(t) \in X'_+$ for all sufficiently large $t$, as claimed.
\end{example}

\begin{example}[Individual eventual nonnegativity does not dualize]
	\label{exa:weak-indiviual-nonnegativity-doesnot-dualize}
	There exists a finite-dimensional real vector space $Y$ ordered by a closed cone $Y_+$ with non-empty interior and an operator $B \in \calL(Y)$ such that the dual semigroup $(e^{tB'})_{t \geq 0}$ is individually eventually nonnegative but the semigroup $(e^{tB})_{t \geq 0}$ itself is not individually eventually nonnegative.

	Indeed, consider the semigroup $(e^{tB})_{t \geq 0}$ on $Y = \calL(X)$ from Example~\ref{exa:individual-vs-weak}, 
	where $X = \bbR^4$.
	Recall that this semigroup is weakly but not individually eventually nonnegative.
	We claim that the dual semigroup $(e^{tB'})_{t \geq 0}$ on $\mathcal{L}(X)'$ is individually eventually nonnegative.%
	\footnote{
		As $\mathcal{L}(X)$ is finite-dimensional and thus canonically isomorphic to its bidual, 
		the bidual semigroup $(e^{tB''})_{t \geq 0}$ can be identified with $(e^{tB})_{t \geq 0}$.
		Hence, this example shows that the dual semigroup of an individually eventually nonnegative semigroup 
		need not be individually eventually nonnegative, in general.
	}
		
	To see this, recall from Proposition~\ref{prop:operator-dual-cone}
	that the dual cone $\mathcal{L}(X)_+'$ of $\mathcal{L}(X)_+$ is the convex hull of
	\begin{align*}
		\big\{ 
			\varphi_{x, x'}
			\mid 
			x \in X_+ \text{ and } x' \in X_+' 
		\big\}
		.
	\end{align*}
	So fix $x \in X_+$ and $x' \in X'_+$. 
	It suffices to show that the operator $e^{tB'} \varphi_{x, x'}$ is positive for all sufficiently large $t$.
	
	A brief computation shows that%
	\footnote{
		As before we let $A \in \bbR^{4 \times 4}$ denote the matrix from Example~\ref{exa:uniform-vs-individual}.
	}
	\begin{align*}
		e^{tB'} \varphi_{x, x'}
		= 
		\varphi_{x, e^{tA'}x'}
		,
	\end{align*}
	for all $t$, 
	and the latter operator is indeed positive for large $t$ 
	since the semigroup $(e^{tA'})_{t \ge 0}$ was shown to be individually eventually nonnegative  
	in Example~\ref{exa:ind-semigroup-and-dual}.
\end{example}

\section{Spectral properties in the weakly eventually nonnegative case}
\label{sec:spectral-properties-weakly-eventually-nonnegative-case}

We now show that, for weakly eventually nonnegative semigroups $(e^{tA})_{t \ge 0}$,
typical Perron--Frobenius (or Krein--Rutman) like properties hold for the \emph{spectral bound}
\begin{align*}
	\spb(A) = \max \{\re \lambda: \; \lambda \in \spec(A)\}
\end{align*}
of $A$ (where $\sigma(A)$ denotes the spectrum of $A$). 
Here we understand the spectrum of an operator $A$ on a finite-dimensional real vector space $X$ 
by extending $A$ to a \emph{complexification} of $X$ and considering the spectrum of this complex extension. 
For $X = \bbR^d$ and $A \in \bbR^{d \times d}$ this simply means to consider $A$ as a matrix in $\bbC^{d \times d}$.

\begin{theorem}\label{thm:spectral-bound}
	Let $\{0\} \not= X$ be a finite-dimensional real vector space, ordered by a closed cone $X_+$ with non-empty interior.
	Let $A \in \calL(X)$ and assume that the semigroup $(e^{tA})_{t \ge 0}$ is weakly eventually nonnegative.
	\begin{enumerate}[\upshape (a)]
		\item\label{thm:spectral-bound:itm:eigenvalue}
		The spectral bound $\spb(A)$ is an eigenvalue of $A$.

		\item\label{thm:spectral-bound:itm:eigenvector}
		There exists an eigenvector $x \in X_+$ of $A$ for the eigenvalue $\spb(A)$.

		\item\label{thm:spectral-bound:itm:dual-eigenvector}
		There exists an eigenvector $x' \in X'_+$ of the dual operator $A'$ for the eigenvalue $\spb(A)$.
	\end{enumerate}
\end{theorem}

Let us recall the important observation that, even in the case of nonnegative semigroups,
one cannot expect the spectral bound $\spb(A)$ to be a \emph{dominant} eigenvalue in the sense that
every other eigenvalue has strictly smaller real part
(for instance, a rotation semigroup on the space $\bbR^3$, endowed with the ice cream cone, gives a counterexample).
For more information on this phenomenon and on the question how it is related to the geometry of the cone,
we refer to \cite{SchneiderVidyasagar1970} and~\cite{Veitsblit1985, VeitsblitLyubich1985}.

We approach the proof of Theorem~\ref{thm:spectral-bound}
from a rather functional analytic perspective
and make heavy use of resolvent operators.
The main ideas are quite similar to arguments used in an (infinite-dimensional)
Banach lattice setting in~\cite[Sections~4--6]{Glueck2017}.

\begin{proof}[Proof of Theorem~\ref{thm:spectral-bound}]
	Throughout the proof we may assume that $X = \bbR^d$ for some integer $d \ge 1$
	and that this space is endowed with the Euclidean norm.%
	\footnote{
		Working on $\bbR^d$ might be a bit more intuitive in this particular proof, 
		since a complexification of $\bbR^d$ is concretely given as $\bbC^d$ --
		this makes the arguments which rely on the complexification a bit less abstract.
	}
	Moreover, by replacing $A$ with $A-\spb(A)$ we may, and shall, 
	assume throughout the proof that $\spb(A) = 0$.
	
	\eqref{thm:spectral-bound:itm:eigenvalue}
	By definition of the spectral bound, there exists a complex number $\lambda$ with real part $\re \lambda = \spb(A) = 0$
	and a non-zero vector $z \in \bbC^d$ such that $Az = \lambda z$.
	Our main line of argument will be to show that, as the resolvent of $A$ blows up at $\lambda$,
	it also has to blow up at $\spb(A)$ due to the weak eventual nonnegativity.
	More precisely, we argue as follows:

	Since the cone $X_+$ is generating in $X = \bbR^d$, 
	the vector $z$ is a complex linear combination of four vectors in $X_+$.
	Moreover, the norm of the vector
	\begin{align*}
		(\mu - A)^{-1} z = \frac{1}{\mu - \lambda} z
	\end{align*}
	explodes as $\mu \in \bbC \setminus \spec(A)$ approaches $\lambda$;
	thus we can find a vector $x \in X_+$ and a real sequence $r_n \downarrow 0$ such that
	\begin{align*}
		\norm{(\lambda + r_n - A)^{-1} x} \to \infty
		\qquad \text{as } n \to \infty.
	\end{align*}
	Next we use that the dual cone $X'_+$ spans the dual space $X' \simeq \bbR^d$
	and hence it also spans $\bbC^d$ when we allow for complex scalars.
	So we can find a functional $x' \in X'_+$ and replace $(r_n)_{n \in \bbN}$ with an appropriate subsequence
	such that
	\begin{align}
		\label{eq:thm:spectral-bound:resolvent-blow-up}
		\modulus{ \langle x', (\lambda + r_n - A)^{-1} x \rangle } \to \infty
		\qquad \text{as } n \to \infty.
	\end{align}
	For the rest of the argument we use the Laplace transform representation of the resolvent of $A$:
	for each $\nu \in \bbC$ of real part $\re \nu > 0 = \spb(A)$ we have
	\begin{align*}
		(\nu - A)^{-1} = \int_0^\infty e^{-t\nu} e^{tA} \dx t.
	\end{align*}
	By the weak eventual nonnegativity of $(e^{tA})_{t \ge 0}$ there exists a time $t_0 \ge 0$ such that
	$\langle x', e^{tA} x \rangle \ge 0$ for all $t \ge t_0$.
	Using the notation $E(\nu) := \int_0^{t_0} e^{-t\nu} e^{tA} \dx t$ for each $\nu \in \bbC$ 
	we obtain for every $n \in \bbN$ 
	\begin{align*}
		 & \modulus{ \langle x', (\lambda + r_n - A)^{-1} x \rangle }
		\le
		\int_{t_0}^\infty e^{-t r_n} \langle x', e^{tA} x \rangle \dx t
		+ \modulus{\langle x', E(\lambda + r_n) x \rangle }
		\\
		 & \quad
		\le
		\modulus{ \langle x', (r_n - A)^{-1} x \rangle }
		+
		\modulus{\langle x', E(r_n) x \rangle }
		+
		\modulus{\langle x', E(\lambda + r_n) x \rangle }
	\end{align*}
	The latter two summands remain bounded as $n \to \infty$,
	so we conclude from the resolvent blow-up in formula~\eqref{eq:thm:spectral-bound:resolvent-blow-up}
	that also
	\begin{align*}
		\modulus{ \langle x', (r_n - A)^{-1} x \rangle } \to \infty
		\qquad \text{as } n \to \infty.
	\end{align*}
	Hence, $\norm{(r_n - A)^{-1}} \to \infty$ as $n \to \infty$,
	which proves that $0$ is in the spectrum of $A$.

	\eqref{thm:spectral-bound:itm:eigenvector}
	Due to~\eqref{thm:spectral-bound:itm:eigenvalue} the resolvent $(\argument - A)^{-1}$ has a pole at $0$;
	let $k \ge 1$ denote its order.
	Then $r^k (r - A)^{-1}$ converges to a non-zero operator $Q$ on $X = \bbR^d$ as $r \downarrow 0$ 
	and every non-zero vector in the range of $Q$ is an eigenvector of $A$ for the eigenvalue $0$
	(this follows, for instance, from \cite[Section~VIII.8, Theorem~2 on p.\;229]{Yosida1980}).

	Let us now show that the operator $Q$ is nonnegative;
	to this end, let $x \in X_+$ and $x' \in X'_+$.
	Due to the weak eventual nonnegativity of the semigroup we find a time $t_0 \ge 0$ such that
	$\langle x', e^{tA} x \rangle \ge 0$ for all $t \ge t_0$.
	Hence, we have
	\begin{align*}
		 &
		\langle x', Qx \rangle
		=
		\lim_{r \downarrow 0}
		r^k \langle x', (r-A)^{-1} x \rangle
		\\
		 & =
		\lim_{r \downarrow 0}
		\Big(
		\underbrace{
				r^k \int_0^{t_0} e^{-tr} \langle x', e^{tA} x \rangle \dx t
			}_{\to 0 \text{ as } r \downarrow 0)}
		+
		\underbrace{r^k}_{\ge 0}
		\int_{t_0}^\infty
		\underbrace{ e^{-tr} \langle x', e^{tA} x \rangle }_{\ge 0}
		\dx t
		\Big),
	\end{align*}
	so $\langle x', Qx \rangle \ge 0$ and thus $Q \ge 0$.

	As $Q$ is non-zero and $X_+$ spans $X$, there exists a vector $x \in X_+$ such that $Qx \not= 0$.
	Hence, $Qx \in X_+$ is an eigenvector of $A$ for the eigenvalue $0$.

	\eqref{thm:spectral-bound:itm:dual-eigenvector}
	By Proposition~\ref{prop:dualize}\eqref{prop:dualize:itm:weak} the dual semigroup $(e^{tA'})_{t \geq 0}$ 
	is weakly eventually nonnegative with respect to the dual cone $X_+'$. 
	Moreover, as discussed at the beginning of Section~\ref{sec:cones-finite-dim-spaces-operator-spaces}, 
	the dual cone $X_+'$ is closed and has non-empty interior. 
	As $A$ and $A'$ have the same eigenvalues one has $\spb(A) = \spb(A')$, 
	so~\eqref{thm:spectral-bound:itm:eigenvalue} and~\eqref{thm:spectral-bound:itm:eigenvector} yield the assertion.
\end{proof}

\section{Eventual positivity}
\label{sec:eventual-positivity}

In this final section we show that eventual positivity behaves much simpler than eventual nonnegativity
in the sense that all three notions (weak, individual, and uniform) are equivalent.
Moreover, the following theorem also shows that these properties can be characterized in purely spectral theoretic terms.
Recall that we call the spectral bound of an operator $A \in \calL(X)$ on a finite-dimensional vector space $X$
a \emph{dominant eigenvalue} of $A$ if it is an eigenvalue of $A$ and every other eigenvalue has strictly smaller real part.
As in the previous section, all spectral theoretic notions are to be understood by considering a complexification of $X$.

\begin{theorem}\label{thm:strong-equiv}
	Let $\{0\} \not= X$ be a finite-dimensional real vector space, ordered by a closed cone $X_+$ with non-empty interior.
	For each $A \in \calL(X)$ the following assertions are equivalent:
	\begin{enumerate}[\upshape (i)]
		\item\label{thm:strong-equiv:itm:unif}
		The semigroup $(e^{tA})_{t \ge 0}$ is uniformly eventually positive.

		\item\label{thm:strong-equiv:itm:ind}
		The semigroup $(e^{tA})_{t \ge 0}$ is individually eventually positive.

		\item\label{thm:strong-equiv:itm:weak}
		The semigroup $(e^{tA})_{t \ge 0}$ is weakly eventually positive.

		\item\label{thm:strong-equiv:itm:geom}
		The spectral bound $\spb(A)$ is a dominant and geometrically simple eigenvalue of $A$;
		the eigenspace $\ker(\spb(A) - A)$ is spanned by an interior point of $X_+$
		and the dual eigenspace $\ker(\spb(A) - A')$ is spanned by an interior point of $X'_+$.

		\item\label{thm:strong-equiv:itm:alg}
		Assertion~\eqref{thm:strong-equiv:itm:geom} holds
		and the eigenvalue $\spb(A)$ of $A$ is even algebraically simple.
	\end{enumerate}
\end{theorem}
\begin{proof}
	There is no loss of generality in assuming that $\spb(A) = 0$.

	``\eqref{thm:strong-equiv:itm:unif}~$\Rightarrow$~\eqref{thm:strong-equiv:itm:ind}''
	This implication is obvious.

	``\eqref{thm:strong-equiv:itm:ind}~$\Rightarrow$~\eqref{thm:strong-equiv:itm:weak}''
	This implication is also immediate, as $\langle x', x \rangle > 0$ for each $x \in \topInt{X_+}$
	and each non-zero $x' \in X'_+$.

	``\eqref{thm:strong-equiv:itm:weak}~$\Rightarrow$~\eqref{thm:strong-equiv:itm:geom}''
	According to Theorem~\ref{thm:spectral-bound} the spectral bound $\spb(A) = 0$ is an eigenvalue of $A$
	with an eigenvector $x_0 \in X_+$;
	similarly, $A'$ has an eigenvector $x_0' \in X'_+$ for the eigenvalue $0$.

	Next we show that every eigenvector $x \in X_+ \cap \ker A$ is an interior point of $X_+$.
	Indeed for such an $x$ and for every non-zero functional $y' \in X'_+$ one has
	\begin{align*}
		\langle y', x \rangle = \langle y', e^{tA} x \rangle > 0
	\end{align*}
	for all sufficiently large times $t$;
	hence, $x$ is indeed an interior point of $X_+$.
	By the same reasoning one can see that every dual eigenvector $x' \in X'_+ \cap \ker A'$
	is an interior point of $X'_+$.
	In particular, it follows that $x_0 \in \topInt{X_+}$ and $x_0' \in \topInt{X_+'}$.

	Now we show that $\ker A$ is one-dimensional.
	To this end, let $0 \not= y \in \ker A$.
	Since $X_+$ does not contain an affine subspace of non-zero dimension,
	there exists a number $\alpha \in \bbR$ such that $x_0 - \alpha y$ is in the topological boundary of $X_+$.
	But since $x_0 - \alpha y \in X_+ \cap \ker A$, this vector is, as shown above,
	either an interior point of $X_+$ or $0$.
	Hence, $x_0 - \alpha y = 0$, which proves that $\ker A$ is indeed one-dimensional.
	The same argument in $X'_+$ shows that $\ker A'$ is also one-dimensional.

	Finally we show that $\spb(A) = 0$ is a dominant spectral value of $A$;
	assume the contrary.
	Then $A$ has an eigenvalue $i \tau \in i\bbR \setminus \{0\}$ with $\tau > 0$,
	so there exists a point $0 \not= z \in X$
	which has a periodic orbit under $(e^{tA})_{t \ge 0}$ with minimal period $\tau > 0$
	(namely, take $z$ to be the real part of an eigenvector of $A$ for the eigenvalue $i\tau$).
	Again as $X_+$ does not contain an affine subspace of non-zero dimension, we can find a number $\alpha \in \bbR$
	such that $x_0 - \alpha z$ is in the topological boundary of $X_+$.
	Hence, there exists a non-zero functional $y' \in X'_+$ such that $\langle y', x_0 - \alpha z \rangle = 0$.
	However, the vector $x_0 - \alpha z$ is non-zero
	(as the minimal period $\tau$ of $z$ is non-zero, while $x_0$ is a fixed point of the semigroup).
	So due to the weak eventual positivity of the semigroup, there exists an integer $n \ge 0$ such that
	\begin{align*}
		0 < \langle y', e^{n\tau A}(x_0 - \alpha z) \rangle = \langle y', x_0 - \alpha z\rangle = 0,
	\end{align*}
	which is a contradiction.

	``\eqref{thm:strong-equiv:itm:geom} $\Rightarrow$ \eqref{thm:strong-equiv:itm:alg}''
	Assume towards a contradiction that the eigenvalue $\spb(A) = 0$ of $A$ is not algebraically simple.%
	\footnote{
		Our argument for algebraic simplicity is loosely inspired
		by an argument in the proof of \cite[Theorem~7]{Cairns2021}.
	}
	Then there exists a vector $x \in X$ such that $A^2 x = 0$ but $Ax \not= 0$.

	According to~\eqref{thm:strong-equiv:itm:geom} $\ker A$ is spanned by an interior point $x_0$ of $X_+$
	and there exists a non-zero functional $x_0' \in \ker(A') \cap X'_+$.
	The vector $Ax$ is a non-zero multiple of $x_0$, so by replacing $x$ with a non-zero multiple, 
	we may assume that $Ax = x_0$.
	Thus we have
	\begin{align*}
		0
		=
		\langle x, A' x_0' \rangle
		=
		\langle A x, x_0' \rangle
		=
		\langle x_0, x_0' \rangle
		>
		0,
	\end{align*}
	which is a contradiction.

	``\eqref{thm:strong-equiv:itm:alg} $\Rightarrow$ \eqref{thm:strong-equiv:itm:unif}''
	One can derive this implication from \cite[Theorem~8~(i)$\Rightarrow$(iv)]{KasigwaTsatsomeros2017},
	but it is also straightforward to give a direct proof instead:

	Let $x_0 \in \ker A$ and $x_0' \in \ker(A')$ be interior points of $X_+$ and $X'_+$, respectively.
	By multiplying one of these vector with a positive scalar we can achieve that $\langle x_0', x_0 \rangle =1$
	and thus the rank-$1$ operator $x_0 \otimes x_0'$ is a projection.
	The properties listed in~\eqref{thm:strong-equiv:itm:alg} and the assumption $\spb(A) = 0$ imply that $e^{tA}$ converges to $x_0 \otimes x_0'$ as $t \to \infty$.
	According to Proposition~\ref{prop:operator-cone-interior} the operator $x_0 \otimes x_0'$
	is an interior point of the cone $\calL(X)_+$.
	Thus, there exists $t_0 \ge 0$ such that $e^{tA}$ is an interior point of $\calL(X)_+$ for each $t \ge t_0$.
	By Proposition~\ref{prop:operator-cone-interior} this means that $e^{tA}$ maps $X_+ \setminus \{0\}$ into $\topInt{X_+}$
	for each $t \ge t_0$, so the semigroup is indeed uniformly eventually positive.
\end{proof}

\appendix

\section{A few notes on the closedness of cones}\label{sec:more-on-closedness-of-cones}

In this appendix we discuss the limitations of Lemma~\ref{lem:closed-convexification}.
We first note that the assumption $\conv(E) \cap -E  = \{0\}$ in Lemma~\ref{lem:closed-convexification} cannot be dropped; 
the following counterexample is essentially taken from~\cite{DeCorte2014}. 
We include it here to be more self-contained, 
in particular since the example serves as a blueprint for the slightly more involved 
Example~\ref{exa:non_convex_cone_4d} below.

\begin{example}\label{exa:non_convex_cone_3d}
	In $\bbR^3$, consider the closed cones
	\begin{align*}
		E_1 
		& := 
		\left\{ 
			x
			\in \bbR^3 
			\mid 
			x_1^2 + x_2^2 \leq x_3^2 \text{ and } x_3 \geq 0 
		\right\}, 
		\\
		E_2 
		& := 
		\left\{ 
			\lambda 
			\begin{psmallmatrix}
				-1 \\ \phantom{-}0 \\ -1
			\end{psmallmatrix}
			\mid 
			\lambda \in [0,\infty) 
		\right\}
	\end{align*}
	and define $E := E_1 \cup E_2$. 
	Then $E$ is closed and one has $\lambda E \subseteq E$ for all $\lambda \geq 0$. 
	Moreover, $\conv(E) = E_1 + E_2$.
	
	One can readily check that the vector $(0, 1, 0)$ is not contained in $\conv(E)$.
	However, this vector is in the closure of $\conv(E)$ since
	\begin{align*}
		\begin{psmallmatrix}
			0 \\ 1 \\ 0
		\end{psmallmatrix}
		= 
		\lim_{\lambda \to \infty} 
		\left[
			\begin{psmallmatrix}
				-\lambda \\ \phantom{-}0 \\ -\lambda
			\end{psmallmatrix}
			+ 
			\begin{psmallmatrix}
				\lambda       \\ 
				1 + \frac{1}{\lambda} \\ 
				\sqrt{\lambda^2 + (1 + \frac{1}{\lambda})^2}
			\end{psmallmatrix}
		\right]
		.
	\end{align*}
	The only assumption of Lemma~\ref{lem:closed-convexification} that is violated 
	is $\conv(E) \cap -E  = \{0\}$. 
	Indeed, the vector $(-1,0,-1)$ is even contained in $E \cap -E$.
\end{example}

In $\bbR^4$ there is even an example of a dilation invariant closed set $E$ 
that satisfies $E \cap -E = \{0\}$ and has non-closed convex hull $\conv(E)$. 
Before we state the example we show that this cannot occur in three dimensions.

\begin{proposition}
	Let $E \subseteq \bbR^3$ be a closed set that satisfies $\lambda E \subseteq E$ for all real numbers $\lambda \geq 0$.
	Assume that $E \cap -E = \{0\}$.
	Then $\conv(E)$ is closed.
\end{proposition}

\begin{proof}
	We use the notation $X := \bbR^3$.
	As the convex hull $\conv(E)$ coincides with the convex cone generated by $E$, 
	a refinement of Caratheodory's theorem for cones, see~\cite[Corollary~IV.17.2 on p.\;156]{Rockafellar1997}, 
	says that each vector in $\conv(E)$ can be written as $y_1 + y_2 + y_3$ for vectors $y_1,y_2,y_2 \in E$,%
	\footnote{
		Recall that in the classical Caratheodory theorem 
		one needs a convex combination of $4 = \dim \bbR^3 + 1$ vectors from $E$.
	}
	i.e., the linear map 
	\begin{align*}
		T: \, X^3 \to X, 
		\quad 
		y \mapsto y_1 + y_2 + y_3
	\end{align*}
	maps $E^3$ surjectively to $\conv(E)$.
	Now let $(x^{(n)})_{n \in \bbN}$ be a sequence in $\conv(E)$ that converges to $x \in \bbR^3$ 
	and choose a sequence $(y^{(n)})_{n \in \bbN}$ in $E^3$ such that $Ty^{(n)} = x^{(n)}$ for each $n \in \bbN$.
	If $(y^{(n)})_{n \in \bbN}$ is bounded in $X^3$ 
	a compactness argument implies that $x \in \conv(E)$ (as in the proof of Lemma~\ref{lem:closed-convexification}).
	
	So assume now that $(y^{(n)})_{n \in \bbN}$ is unbounded. 
	We will show that this implies directly that $\conv(E)$ is closed.
	As in the proof of Lemma~\ref{lem:closed-convexification}, we 
	replace $(y^{(n)})_{n \in \bbN}$ -- and, accordingly, $(x^{(n)})_{n \in \bbN}$ -- 
	with a subsequence such that $0 < \norm{y^{(n)}} \to \infty$ 
	and such that $(y^{(n)} / \norm{y^{(n)}})_{n \in \bbN}$ converges 
	to a vector $y \in E^3$ of norm $1$.
	At least one of the components of $y$, say $y_1$, is a non-zero vector in $X$.
	
	We have $Ty = \lim_n x^{(n)} / \norm{y^{(n)}} = 0$, i.e., $y_1+y_2+y_3 = 0$.
	It follows that $y_2 \not= 0$, since otherwise $0 \not= y_1 = -y_3 \in E \cap -E$; 
	the same argument shows that $y_3 \not= 0$.
	So the vectors $y_1,y_2,y_3$ are all non-zero and sum up $0$, 
	which implies that the wedge $W := [0,\infty)y_1 + [0,\infty)y_2 + [0,\infty)y_3$ 
	spanned by them is a two dimensional vector subspace of $\bbR^3$. 
	
	So the wedge $\conv(E)$ in $\bbR^3$ contains the two dimensional vector subspace $W$ 
	and thus $\conv(E)$ can only be one of the following sets: 
	the set $W$, or a half space on one side of $W$ that contains $W$, or $\bbR^3$. 
	In each of those cases $\conv(E)$ is closed.
\end{proof}

The following modification of Example~\ref{exa:non_convex_cone_3d} 
yields a dilation invariant set $E \subseteq \bbR^4$ 
that satisfies $E \cap -E = \{0\}$ and has non-closed convex hull $\conv(E)$.

\begin{example}
	\label{exa:non_convex_cone_4d}
	In $\bbR^4$, consider the closed cones
	\begin{align*}
		E_1 
		& := 
		\left\{x \in \bbR^4 : \, x_1^2 + x_2^2 + x_3^2  \leq x_4^2 \text{ and } x_4 \geq 0 \right\}
		, 
		\\ 
		E_2 
		& := 
		\left\{ 
			\lambda 
			\begin{psmallmatrix}
				-1 \\ -1 \\ \phantom{-}0 \\ -1
			\end{psmallmatrix}
			\mid 
			\lambda \in [0,\infty) 
		\right\}
		,
		\\
		E_3 
		& := 
		\left\{ 
			\lambda 
			\begin{psmallmatrix}
				-1 \\ \phantom{-}1 \\ \phantom{-}0 \\ -1
			\end{psmallmatrix}
			\mid 
			\lambda \in [0,\infty) 
		\right\}
	\end{align*}
	and set
	\begin{align*}
		E := E_1 \cup E_2 \cup E_3.
	\end{align*}
	Then $E$ is closed and satisfies $\lambda E \subseteq E$ for all $\lambda \geq 0$ as well as $E \cap - E = \{0\}$. 
	Moreover, $\conv(E) = E_1 + E_2 + E_3$.
	
	The vector $(0, 0, 1, 0)$ is not contained in $\conv(E)$. 
	Indeed, if $x \in E_1$ and $\lambda_1, \lambda_2 \in [0,\infty)$ such that 
	\begin{align*}
		\begin{psmallmatrix}
			0 \\ 0 \\ 1 \\ 0
		\end{psmallmatrix}
		= 
		x 
		\; + \; 
		\lambda_1 
		\begin{psmallmatrix}
			-1 \\ -1 \\ \phantom{-}0 \\ -1
		\end{psmallmatrix}
		+ 
		\lambda_2 
		\begin{psmallmatrix}
			-1 \\ \phantom{-}1 \\ \phantom{-}0 \\ -1
		\end{psmallmatrix}
		,
	\end{align*}
	then $x_1 = x_4$ and $x_3 = 1$, which contradicts $x \in E_1$.
	However, the vector $(0,0,1,0)$ is in the closure of $\conv(E)$, since
	\begin{align*}
		\begin{psmallmatrix}
			0 \\ 0 \\ 1 \\ 0
		\end{psmallmatrix}
		= 
		\lim_{\lambda \to \infty}
		\left[ 
			\frac{\lambda}{2} 
			\begin{psmallmatrix}
				-1 \\ -1 \\ \phantom{-}0 \\ -1
			\end{psmallmatrix}
			+
			\frac{\lambda}{2} 
			\begin{psmallmatrix}
				-1 \\ \phantom{-}1 \\ \phantom{-}0 \\ -1
			\end{psmallmatrix}
			+
			\begin{psmallmatrix}
				\lambda     \\ 
				0           \\
				1 + \frac{1}{\lambda} \\ 
				\sqrt{\lambda^2 + \left( 1 + \frac{1}{\lambda} \right)^2}
			\end{psmallmatrix}
		\right] 
		.
	\end{align*}
\end{example}

\bibliographystyle{plain}
\bibliography{literature}

\end{document}